\documentclass{article}
\usepackage{amsmath} 
\usepackage{mathtools} 
\usepackage{tikz}
\usepackage{float} 
\usepackage[utf8]{inputenc}
\usepackage{comment} 
\usepackage{amsthm} 
\usetikzlibrary{decorations.pathreplacing} 
\usepackage[colorlinks=true]{hyperref}%

\newtheorem{theorem}{Theorem}[section] 
\newtheorem{lemma}[theorem]{Lemma} 
\newtheorem{conjecture}[theorem]{Conjecture} 

\providecommand{\keywords}[1]{\textbf{\textit{Keywords---}} #1} 

\title{The independence polynomial of trees is not always log-concave starting from order 26}

\author{
        Ohr Kadrawi \\
        Department of Mathematics\\
        Ariel University\\
        Ariel 4070000, \underline{Israel}\\
        orka@ariel.ac.il
        \and
        Vadim E. Levit \\
        Department of Mathematics\\
        Ariel University\\
        Ariel 4070000, \underline{Israel}\\
        levitv@ariel.ac.il
}
\date{}

\begin{document}

\maketitle

\begin{abstract}
    An independent set in a graph is a collection of vertices that are not adjacent to each other. The cardinality of the largest independent set in $G$ is represented by $\alpha(G)$. The independence polynomial of a graph $G = (V, E)$ was introduced by Gutman and Harary in 1983 and is defined as
    \[
    I(G;x) = \sum_{k=0}^{\alpha(G)}{s_k}x^{k}={s_0}+{s_1}x+{s_2}x^{2}+...+{s_{\alpha(G)}}x^{\alpha(G)},
    \]
    where $s_k$ represents the number of independent sets in $G$ of size $k$. The conjecture made by Alavi, Malde, Schwenk, and Erd{\"o}s in 1987 stated that the independence polynomials of trees are unimodal, and many researchers believed that this conjecture could be strengthened up to its corresponding log-concave version. However, in our paper, we present evidence that contradicts this assumption by introducing infinite families of trees whose independence polynomials are not log-concave.

\end{abstract}


\section{Introduction}

In this paper, we consider a simple graph $G = (V, E)$, which is a finite, undirected graph without any loops or multiple edges. The vertex set of $G$ is denoted by $V(G)$ and has cardinality $|V(G)| = n(G)$, while the edge set is denoted by $E(G)$ and has cardinality $|E(G)| = m(G)$.\\

For a given vertex $v \in V$, its \textit{neighborhood} is defined as the set of vertices adjacent to it, denoted as $N_G(v) = \{u: u \in V$ and $uv \in E\}$. The notation $N_G[v]$ denotes the closed neighborhood of $v$, which includes the vertex $v$ itself, such that $N_G[v] = N_G(v) \cup {v}$\\

The graph obtained by taking the \textit{disjoint union} of two graphs $G_1$ and $G_2$ is denoted by $G = G_1 \cup G_2$. This graph has a vertex set that is the disjoint union of the vertex sets of $G_1$ and $G_2$, and its edge set is the disjoint union of the edge sets of $G_1$ and $G_2$.\\

An \textit{independent} set or a \textit{stable set} in a graph $G$ is a set of vertices that are not adjacent to each other. A \textit{maximum independent set} is an independent set of the largest possible size in $G$. The cardinality of a maximum independent set is denoted by $\alpha(G)$ and is referred to as the \textit{independence number} of $G$.\\

The number of independent sets of cardinality $k$ in a graph $G$ is denoted by $s_k$. For instance, $s_0=1$ corresponds to the number of independent sets of minimum cardinality in $G$, which is represented by the empty set. The polynomial
\[
    I(G;x) = \sum_{k=0}^{\alpha(G)}{s_k}x^{k}={s_0}+{s_1}x+{s_2}x^{2}+...+{s_{\alpha(G)}}x^{\alpha(G)},
\]
 is known as the \textit{independence polynomial} of the graph $G$ \cite{GutmanHarary1983}, it is also referred to as the independent set polynomial \cite{Hoede_Li1994} or the stable set polynomial \cite{ChudnovskySeymour2006}. Recent developments related to the independence polynomial can be found in \cite{BallGalvinHyryWeingartner2021, BasitGalvin2021}.\\

The independence polynomial $I(G;x)$ of a graph $G$ can be computed using the recursive formula shown in \cite{LevitMandrescu2005} and also in \cite{Arocha1984, GutmanHarary1983, Hoede_Li1994}:
        \begin{equation}
        \label{eq:independent-polynomial-rec}
            I(G; x) = I(G - v;x) + x \cdot I(G-N[v]; x).
        \end{equation}
        
For the disjoint union of graphs, the independence polynomial can be computed using the formula (\cite{Hoede_Li1994, LevitMandrescu2005}):
        \begin{equation}
        \label{eq:independent-polynomial-graphs}
            I(G_1 \cup G_2; x) = I(G_1;x) \cdot I(G_2;x).
        \end{equation}

A finite sequence of real numbers $(a_0, a_1, a_2, \dots , a_n)$ is said to be:
\begin{itemize}
    \item \textit{unimodal} if there exists a unique index $k \in {0, 1, \dots , n}$, called the \textit{mode} of the sequence, such that
    \[
    a_0 \leq \dots \leq a_{k-1} \leq a_k \geq a_{k+1} \geq \dots \geq a_n;
    \]
    the mode is \textit{unique} if $a_{k-1} < a_k > a_{k+1}$;
    \item \textit{logarithmically concave} (or simply, \textit{log-concave}) if the inequality
    \[
    a_k^2 \geq a_{k-1} \cdot a_{k+1}
    \]
    is satisfied for all $k \in {1, 2, \dots, n-1}$.

\end{itemize}

It is known that any log-concave sequence of positive numbers is also unimodal. Unimodal and log-concave sequences occur in many areas of mathematics, including algebra, combinatorics, graph theory, and geometry. For example, the sequence of binomial coefficients, presented in the $n$th row of Pascal's triangle is log-concave, for more uses see \cite{BenderCanfield1996, BorosMoll1999, Brenti1990, Brenti1994, BrownColbourn1994, Dukes2002, Stanley1989}. A considerable amount of literature has been published on the unimodality and the log-concavity of various polynomials defined on graphs \cite{BeatonBrown2022, ChudnovskySeymour2006, Hamidoune1990, Horrocks2002, LevitMandrescu2003, Schwenk1981}.\\

Alavi, Malde, Schwenk, and Erd{\"o}s conjectured that the independence polynomials of trees are unimodal \cite{AlaviMaldeSchwenkErdos1987}. Yosef, Mizrachi, and Kadrawi developed a method for computing the independence polynomials of trees on $n$ vertices using a database \cite{YosefMizrachiKadrawi2021}. It helped to validate both unimodality and log-concavity of independence polynomials of trees with up to $20$ vertices. Moreover, Radcliffe further verified that independence polynomials of trees up to $25$ vertices are log-concave \cite{Radcliffe}.\\

Over the years, there have been several attempts to extend the above conjecture. The first direction was for unimodality. In 2006, Levit and Mandrescu conjectured that the independence polynomials of K\"{o}nig-Egerv\'{a}ry graphs are unimodal \cite{LevitMandrescu2006}. In 2013, Bhattacharya and Kahn constructed a bipartite graph with non-unimodal independence polynomial and reported that this is the smallest counterexample that exists \cite{BhattacharyaKahn2013}. However, after two years, Schwenk found some additional examples that were even smaller than the graph constructed by Bhattacharya and Kahn \cite{Schwenk2015}.\\

The second direction was for log-concavity. In 2004, Levit and Mandrescu conjectured that the independence polynomials of every forest are log-concave \cite{LevitMandrescu2004}. In 2011, Galvin also suggested strengthening the unimodality conjecture to a corresponding log-concavity conjecture for trees, forests, and bipartite graphs \cite{Galvin2011}. As of 2023, the conjecture still remained open \cite{XieFengXu2023}. However, in the same year, Kadrawi, Levit, Yosef, and Mizrachi found a number of counterexamples to the log-concavity conjecture in \cite{KadrawiLevitYosefMizrachi2023}. This paper begins by using the counterexamples presented in \cite{KadrawiLevitYosefMizrachi2023} and expands upon them to create multiple new infinite families of counterexamples. These families significantly broaden the scope of the counterexamples found before and give  hope to finding all the counterexamples that are possible.

\section{Basic counterexamples}
To support the Alavi, Malde, Schwenk, and Erd{\"o}s conjecture \cite{AlaviMaldeSchwenkErdos1987}, it was verified that for trees with up to $25$ vertices, their independence polynomials are log-concave and, therefore, unimodal. This was also demonstrated in \cite{Radcliffe}. However, when the number of vertices in a tree reached $26$, all the trees were found to have unimodal independence polynomials, while only two trees were discovered to have non-log-concave independence polynomials \cite{KadrawiLevitYosefMizrachi2023}, which served as counterexamples to the conjectures made by Levit and Mandrescu \cite{LevitMandrescu2004} and Galvin \cite{Galvin2011}.

\begin{figure}[H]
    \centering
    \begin{tikzpicture}[scale=0.4]
        \tikzstyle{black}=[fill=black, draw=black, shape=circle, inner sep = 2]
        
		\node [style=black] (0) at (0, 2) {};
		\node [style=black] (1) at (1, 2) {};
		\node [style=black] (2) at (2, 2) {};
		\node [style=black] (3) at (0, 3) {};
		\node [style=black] (4) at (1, 3) {};
		\node [style=black] (5) at (2, 3) {};
		\node [style=black] (6) at (1, 5) {};
		\node [style=black] (7) at (3.25, 2) {};
		\node [style=black] (8) at (4.25, 2) {};
		\node [style=black] (9) at (5.25, 2) {};
		\node [style=black] (10) at (3.25, 3) {};
		\node [style=black] (11) at (4.25, 3) {};
		\node [style=black] (12) at (5.25, 3) {};
		\node [style=black] (13) at (4.75, 5) {};
		\node [style=black] (14) at (7.5, 2) {};
		\node [style=black] (15) at (8.5, 2) {};
		\node [style=black] (16) at (9.5, 2) {};
		\node [style=black] (17) at (7.5, 3) {};
		\node [style=black] (18) at (8.5, 3) {};
		\node [style=black] (19) at (9.5, 3) {};
		\node [style=black] (20) at (9, 5) {};
		\node [style=black] (21) at (4.75, 7) {};
		\node [style=black] (24) at (10.5, 2) {};
		\node [style=black] (25) at (10.5, 3) {};
		\node [style=black] (26) at (6.25, 2) {};
		\node [style=black] (27) at (6.25, 3) {};
        \node (28) at (0, 0) {};
		\node (29) at (4.75, 0) {$T_1$};

		\draw (3) to (0);
		\draw (4) to (1);
		\draw (5) to (2);
		\draw (6) to (5);
		\draw (6) to (4);
		\draw (6) to (3);
		\draw (10) to (7);
		\draw (11) to (8);
		\draw (12) to (9);
		\draw (13) to (12);
		\draw (13) to (11);
		\draw (13) to (10);
		\draw (20) to (17);
		\draw (20) to (18);
		\draw (20) to (19);
		\draw (17) to (14);
		\draw (18) to (15);
		\draw (19) to (16);
		\draw (21) to (20);
		\draw (21) to (13);
		\draw (21) to (6);
		\draw (20) to (25);
		\draw (25) to (24);
		\draw (13) to (27);
		\draw (27) to (26);
\end{tikzpicture}
\hspace{1cm}
\begin{tikzpicture}[scale=0.4]
        \tikzstyle{black}=[fill=black, draw=black, shape=circle, inner sep = 2]
		\node [style=black] (0) at (0, 2) {};
		\node [style=black] (1) at (1, 2) {};
		\node [style=black] (2) at (2, 2) {};
		\node [style=black] (3) at (0, 3) {};
		\node [style=black] (4) at (1, 3) {};
		\node [style=black] (5) at (2, 3) {};
		\node [style=black] (6) at (1, 5) {};
		\node [style=black] (7) at (3.25, 2) {};
		\node [style=black] (8) at (4.25, 2) {};
		\node [style=black] (9) at (5.25, 2) {};
		\node [style=black] (10) at (3.25, 3) {};
		\node [style=black] (11) at (4.25, 3) {};
		\node [style=black] (12) at (5.25, 3) {};
		\node [style=black] (13) at (4.25, 5) {};
		\node [style=black] (14) at (6.5, 2) {};
		\node [style=black] (15) at (7.5, 2) {};
		\node [style=black] (16) at (8.5, 2) {};
		\node [style=black] (17) at (6.5, 3) {};
		\node [style=black] (18) at (7.5, 3) {};
		\node [style=black] (19) at (8.5, 3) {};
		\node [style=black] (20) at (8, 5) {};
		\node [style=black] (21) at (4.25, 7) {};
		\node [style=black] (22) at (0, 1) {};
		\node [style=black] (23) at (0, 0) {};
		\node [style=black] (24) at (9.5, 2) {};
		\node [style=black] (25) at (9.5, 3) {};
		
		\node (26) at (4.25, 0) {$T_2$};

		\draw (3) to (0);
		\draw (4) to (1);
		\draw (5) to (2);
		\draw (6) to (5);
		\draw (6) to (4);
		\draw (6) to (3);
		\draw (10) to (7);
		\draw (11) to (8);
		\draw (12) to (9);
		\draw (13) to (12);
		\draw (13) to (11);
		\draw (13) to (10);
		\draw (20) to (17);
		\draw (20) to (18);
		\draw (20) to (19);
		\draw (17) to (14);
		\draw (18) to (15);
		\draw (19) to (16);
		\draw (21) to (20);
		\draw (21) to (13);
		\draw (21) to (6);
		\draw (0) to (22);
		\draw (22) to (23);
		\draw (20) to (25);
		\draw (25) to (24);

\end{tikzpicture}
    \caption{Two trees with 26 vertices whose independence polynomials are not log-concave.}
    \label{fig:two_trees}

\end{figure}
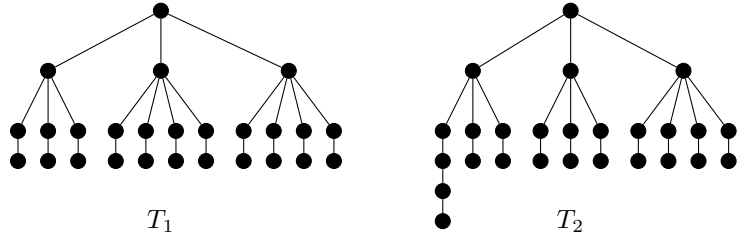
The independence polynomials of the trees $T_1$ and $T_2$ defined in \textbf{Figure \ref{fig:two_trees}} are as follows: 
\[
I(T_1;x) =\!\begin{multlined}[t]
    x^{14} + 51x^{13} + 2979x^{12} + 18683x^{11} + 55499x^{10} + 100144x^9 + \\
    121376x^8 + 103736x^7 + 63933x^6 + 28551x^5 + 9142x^4 + 2040x^3 + \\
    300x^2 + 26x + 1,
\end{multlined}
\]
where the non-log-concavity is demonstrated by the coefficient of $x^{13}$:\\ $51^2=2601<2979$, and
\[
I(T_2;x) =\!\begin{multlined}[t]
     x^{14} + 48x^{13} + 2372x^{12} + 15498x^{11} + 48086x^{10} + 90178x^9 + \\
    112870x^8 + 98968x^7 + 62183x^6 + 28147x^5 + 9089x^4 + 2037x^3 + \\
    300x^2 + 26x + 1,
    \end{multlined}
\]
where the non-log-concavity is demonstrated by the coefficient of $x^{13}$: \\$48^2=2304<2372$.\\

\section{Extensions of T1}
\subsection{3,k,k structure}
The tree labeled as $T_1$ in \textbf{Figure \ref{fig:two_trees}} is part of a larger family of trees with the same property, namely, that their independence polynomials are not log-concave. These trees have a particular structure which we refer to as the \textit{3,k,k structure}, and it is depicted in \textbf{Figure \ref{fig:tree_3_k_K}} and explained below:
\begin{itemize}
    \item the tree has one center, denoted $v_0$ that is connected to three vertices $v_1, v_2, v_3$ 
    \item $v_1$ is connected to $K_2 \cup K_2 \cup K_2$
    \item $v_2$ is connected to $K_2 \cup \dots \cup K_2$, $k$ times
    \item $v_3$ is also connected to another $K_2 \cup \dots \cup K_2$, $k$ times
\end{itemize}

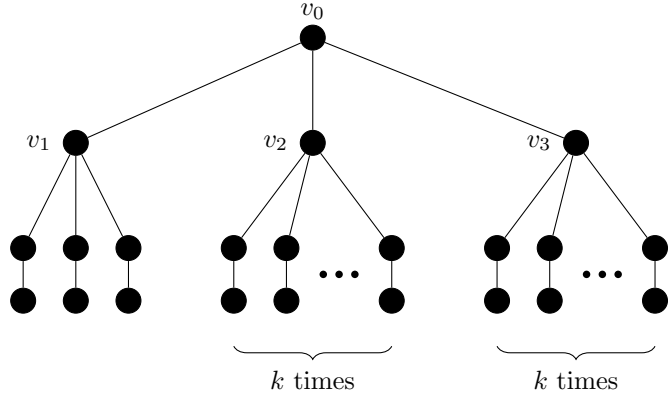
\begin{figure}[H]
    \centering
    \begin{tikzpicture}[scale=0.7]
        \tikzstyle{black}=[fill=black, draw=black, shape=circle]
        
		\node [style=black] (0) at (0, 0) {};
		\node [style=black] (1) at (1, 0) {};
		\node [style=black] (2) at (2, 0) {};
		\node [style=black] (3) at (0, 1) {};
		\node [style=black] (4) at (1, 1) {};
		\node [style=black] (5) at (2, 1) {};
		\node [style=black] (6) at (1, 3) {};
		\node [style=black] (7) at (4, 0) {};
		\node [style=black] (8) at (5, 0) {};
		\node [style=black] (9) at (7, 0) {};
		\node [style=black] (10) at (4, 1) {};
		\node [style=black] (11) at (5, 1) {};
		\node [style=black] (12) at (7, 1) {};
		\node [style=black] (13) at (5.5, 3) {};
		\node [style=black] (14) at (9, 0) {};
		\node [style=black] (15) at (10, 0) {};
		\node [style=black] (16) at (12, 0) {};
		\node [style=black] (17) at (9, 1) {};
		\node [style=black] (18) at (10, 1) {};
		\node [style=black] (19) at (12, 1) {};
		\node [style=black] (20) at (10.5, 3) {};
		\node [style=black] (21) at (5.5, 5) {};
		\draw node[fill,circle, inner sep=1] at (5.7,0.5) {};
		\draw node[fill,circle, inner sep=1] at (6,0.5) {};
		\draw node[fill,circle, inner sep=1] at (6.3,0.5) {};
		\draw node[fill,circle, inner sep=1] at (10.7,0.5) {};
		\draw node[fill,circle, inner sep=1] at (11,0.5) {};
		\draw node[fill,circle, inner sep=1] at (11.3,0.5) {};
		
        \draw [decorate,decoration={brace,amplitude=5pt,mirror,raise=4ex}] (4,0) -- (7,0) node[midway,yshift=-3em]{$k$ times};
        
        \draw [decorate,decoration={brace,amplitude=5pt,mirror,raise=4ex}] (9,0) -- (12,0) node[midway,yshift=-3em]{$k$ times};
		\draw (3) to (0);
		\draw (4) to (1);
		\draw (5) to (2);
		\draw (6) to (5);
		\draw (6) to (4);
		\draw (6) to (3);
		\draw (10) to (7);
		\draw (11) to (8);
		\draw (12) to (9);
		\draw (13) to (12);
		\draw (13) to (11);
		\draw (13) to (10);
		\draw (20) to (17);
		\draw (20) to (18);
		\draw (20) to (19);
		\draw (17) to (14);
		\draw (18) to (15);
		\draw (19) to (16);
		\draw (21) to (20);
		\draw (21) to (13);
		\draw (21) to (6);
		\node at (5.5,5.5) {$v_0$};
		\node at (0.3,3) {$v_1$};
		\node at (4.8,3) {$v_2$};
		\node at (9.8,3) {$v_3$};
\end{tikzpicture}

    \caption{An illustration of the 3,k,k structure of trees that have non-log-concave independence polynomials. }
    \label{fig:tree_3_k_K}
\end{figure}

\begin{lemma}[\cite{KadrawiLevitYosefMizrachi2023}]
    All trees of the $3,k,k$ structure, where $k \geq 4$, have non-log-concave independence polynomials.
\end{lemma}

\subsection{3,k,k+1 structure}
We now consider the possibility of expanding the graph by adding another $K_2$ to one of the clusters. This leads to a new structure, which is described below and illustrated in Figure \ref{fig:tree_3_k_K+1}:
\begin{itemize}
    \item the tree has one center, called $v_0$ that connect to three vertices $v_1, v_2, v_3$ 
    \item $v_1$ connect to $K_2 \cup K_2 \cup K_2$
    \item $v_2$ connect to $K_2 \cup \dots \cup K_2$, $k$ times
    \item $v_3$ also connect to another $K_2 \cup \dots \cup K_2$, $k+1$ times
\end{itemize}

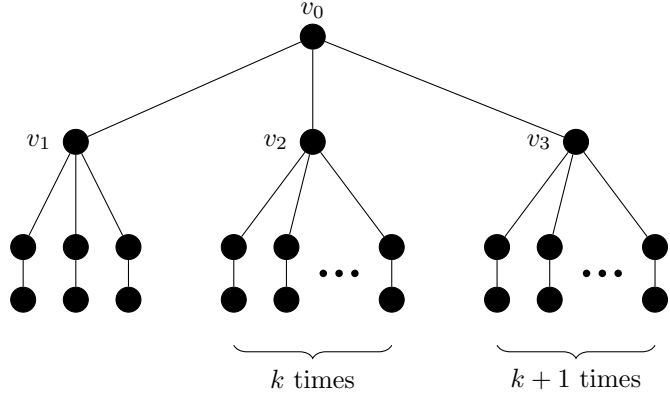
\begin{figure}[H]
    \centering
    \begin{tikzpicture}[scale=0.7]
        \tikzstyle{black}=[fill=black, draw=black, shape=circle]
        
		\node [style=black] (0) at (0, 0) {};
		\node [style=black] (1) at (1, 0) {};
		\node [style=black] (2) at (2, 0) {};
		\node [style=black] (3) at (0, 1) {};
		\node [style=black] (4) at (1, 1) {};
		\node [style=black] (5) at (2, 1) {};
		\node [style=black] (6) at (1, 3) {};
		\node [style=black] (7) at (4, 0) {};
		\node [style=black] (8) at (5, 0) {};
		\node [style=black] (9) at (7, 0) {};
		\node [style=black] (10) at (4, 1) {};
		\node [style=black] (11) at (5, 1) {};
		\node [style=black] (12) at (7, 1) {};
		\node [style=black] (13) at (5.5, 3) {};
		\node [style=black] (14) at (9, 0) {};
		\node [style=black] (15) at (10, 0) {};
		\node [style=black] (16) at (12, 0) {};
		\node [style=black] (17) at (9, 1) {};
		\node [style=black] (18) at (10, 1) {};
		\node [style=black] (19) at (12, 1) {};
		\node [style=black] (20) at (10.5, 3) {};
		\node [style=black] (21) at (5.5, 5) {};
		
		\draw node[fill,circle, inner sep=1] at (5.7,0.5) {};
		\draw node[fill,circle, inner sep=1] at (6,0.5) {};
		\draw node[fill,circle, inner sep=1] at (6.3,0.5) {};
		\draw node[fill,circle, inner sep=1] at (10.7,0.5) {};
		\draw node[fill,circle, inner sep=1] at (11,0.5) {};
		\draw node[fill,circle, inner sep=1] at (11.3,0.5) {};
		
        \draw [decorate,decoration={brace,amplitude=5pt,mirror,raise=4ex}] (4,0) -- (7,0) node[midway,yshift=-3em]{$k$ times};
        
        \draw [decorate,decoration={brace,amplitude=5pt,mirror,raise=4ex}] (9,0) -- (12,0) node[midway,yshift=-3em]{$k+1$ times};
		\draw (3) to (0);
		\draw (4) to (1);
		\draw (5) to (2);
		\draw (6) to (5);
		\draw (6) to (4);
		\draw (6) to (3);
		\draw (10) to (7);
		\draw (11) to (8);
		\draw (12) to (9);
		\draw (13) to (12);
		\draw (13) to (11);
		\draw (13) to (10);
		\draw (20) to (17);
		\draw (20) to (18);
		\draw (20) to (19);
		\draw (17) to (14);
		\draw (18) to (15);
		\draw (19) to (16);
		\draw (21) to (20);
		\draw (21) to (13);
		\draw (21) to (6);
		\node at (5.5,5.5) {$v_0$};
		\node at (0.3,3) {$v_1$};
		\node at (4.8,3) {$v_2$};
		\node at (9.8,3) {$v_3$};
\end{tikzpicture}

    \caption{Illustration of 3,k,k+1 structure of trees that have non-log-concave independence polynomials. }
    \label{fig:tree_3_k_K+1}
\end{figure}

\begin{theorem}
    All trees from $3,k,k+1$ structure, where $k \geq 4$, have non-log-concave independence polynomials. 
\end{theorem}

\begin{proof}
Let compute the independence polynomial $3,k,k+1$-structure, and choose $v_0$ to be the first vertex to remove from the graph.
\[I(G)_{v_0} = I(G-v_0)+x \cdot I(G-N[v_0])\]
Expand the first term in that sum:
\[
\begin{split}
I(G-v_0) &= \!\begin{multlined}[t]
[(2x+1)^{k+1}+x (x+1)^{k+1}] \cdot [(2x+1)^k+x (x+1)^k] \cdot \\
[(2x+1)^3+x (x+1)^3]
\end{multlined}\\
&= \!\begin{multlined}[t]
\bigg[\sum_{i=0}^{k+1} \binom{k+1}{i} (2x)^{i} +x\sum_{i=0}^{k+1} \binom{k+1}{i} x^{i}\bigg] \cdot \\ 
\bigg[\sum_{i=0}^{k} \binom{k}{i} (2x)^{i} +x\sum_{i=0}^{k} \binom{k}{i} x^{i}\bigg] \\
\cdot [(2x+1)^3+x (x+1)^3]
\end{multlined}\\
&= \!\begin{multlined}[t]
[x^{k+2}+(2^{k+1}+k+1)x^{k+1}+((k+1)\cdot2^k+\frac{k(k+1)}{2})x^k+\dots] \cdot \\
[x^{k+1}+(2^k+k)x^k+(k\cdot2^{k-1}+\frac{k(k-1)}{2})x^{k-1}+\dots] \cdot \\
[x^4+11x^3+15x^2+ \dots]
\end{multlined}\\
\end{split}
\]

We can divide the first term expression into three factors $A \cdot B \cdot C$ such that:
\begin{itemize}
    \item $A = [x^{k+2}+(2^{k+1}+k+1)x^{k+1}+((k+1)\cdot2^k+\frac{k(k+1)}{2})x^k+\dots]$
    \item $B = [x^{k+1}+(2^k+k)x^k+(k\cdot2^{k-1}+\frac{k(k-1)}{2})x^{k-1}+\dots]$
    \item $C = [x^4+11x^3+15x^2+ \dots]$
\end{itemize}

Expand the second term in that sum:
\[
\begin{split}
I(G-N[v_0]) &= x \cdot (2x+1)^{2k+4}\\
&= x \cdot \bigg[\sum_{i=0}^{2k+4} \binom{2k+4}{i} (2x)^{i}\bigg]\\
&= x \cdot [(2k)^{2k+4}+\dots]\\
&= 2^{2k+4}x^{2k+5}+\dots
\end{split}
\]

The highest exponent that can reach is $2k+7$, from taking the biggest exponent from every factor, from factor $A$ we multiply by $x^{k+2}$, from factor $B$ we multiply by $x^{k+1}$ and from $C$ factor we multiply by $x^4$, so 
\[
x^{k+2} \cdot x^{k+1} \cdot x^4 = x^{2k+7},
\]
where the final coefficient is equal to $1$.\\

To calculate the coefficient of $X^{2k+6}$ we have 3 options:
\begin{itemize}
    \item multiply $x^{k+2}$ from $A$, $x^{k+1}$ from $B$ and $11x^3$ from $C$
    \item multiply $x^{k+2}$ from $A$, $(2^k+k)x^k$ from $B$ and $x^4$ from $C$
    \item multiply $(2^{k+1}+k+1)x^{k+1}$ from $A$, $x^{k+1}$ from $B$ and $x^4$ from $C$
\end{itemize}

\[
\!\begin{multlined}[t]
x^{k+2} \cdot x^{k+1} \cdot 11x^3 + x^{k+2} \cdot (2^k+k)x^k \cdot x^4 + (2^{k+1}+k+1)x^{k+1} \cdot x^{k+1} \cdot x^4 =\\
(3\cdot2^k + 2k + 12) x^{2k + 6},
\end{multlined}
\]

where the final coefficient is equal to $(3\cdot 2^k +2k+ 12)$.\\

To calculate the coefficient of $x^{2k+5}$ we have 6 options:

\begin{itemize}
    \item multiply $x^{k+2}$ from $A$, $x^{k+1}$ from $B$ and $15x^2$ from $C$
    \item multiply $x^{k+2}$ from $A$, $(2^k+k)x^k$ from $B$ and $11x^3$ from $C$
    \item multiply $x^{k+2}$ from $A$, $(k\cdot2^{k-1}+\frac{k(k-1)}{2})x^{k-1}$ from $B$ and $x^4$ from $C$
    \item multiply $(2^{k+1}+k+1)x^{k+1}$ from $A$, $x^{k+1}$ from $B$ and $11x^3$ from $C$
    \item multiply $(2^{k+1}+k+1)x^{k+1}$ from $A$, $(2^k+k)x^k$ from $B$ and $x^4$ from $C$
    \item multiply $((k+1)\cdot2^k+\frac{k(k+1)}{2})x^k$ from $A$, $x^{k+1}$ from $B$ and $x^4$ from $C$
\end{itemize} 
and also we add the coefficient $2^{2k+4}$ from the second term.

\[
\begin{multlined}
x^{k+2} \cdot x^{k+1} \cdot 15x^2 + 
x^{k+2} \cdot (2^k+k)x^k \cdot 11x^3 + \\
x^{k+2} \cdot (k\cdot2^{k-1}+\frac{k(k-1)}{2})x^{k-1} \cdot x^4 + \\
(2^{k+1}+k+1)x^{k+1} \cdot x^{k+1} \cdot 11x^3 + 
(2^{k+1}+k+1)x^{k+1} \cdot (2^k+k)x^k \cdot x^4 + \\
((k+1)\cdot2^k+\frac{k(k+1)}{2})x^k \cdot x^{k+1} \cdot x^4 + 
2^{2k+4}x^{2k+5}=\\
2\cdot k^2 + 23k + 9\cdot2^{2 k + 1} + 2^{k - 1} (9 k + 70) + 26,
\end{multlined}
\]

where the final coefficient is equal to $2\cdot k^2 + 23k + 9\cdot2^{2 k + 1} + 2^{k - 1} (9 k + 70) + 26$.\\

So the independence polynomial is:

\begin{multlined}[t]
I(G) = x^{2k+7} + (3\cdot 2^k +2k+ 12)x^{2k+6} + 
\\
[2\cdot k^2 + 23k + 9\cdot2^{2 k + 1} + 2^{k - 1} (9 k + 70) + 26]x^{2k+5}+\dots
\end{multlined}

Now, lets prove the non-log-concave of the $x^{2k+6}$ term in the independence polynomial:
\[
(3\cdot 2^k +2k+ 12)^2 < 1 \cdot [2\cdot k^2 + 23k + 9\cdot2^{2 k + 1} + 2^{k - 1} (9 k + 70) + 26]
\]
The left-hand side and the right-hand side equals $k \approx 3.2329$ so the left-hand side is smaller than the right-hand side from $k=4$ and above.
\end{proof}

\subsection{3,k,k+2 structure}
Can this structure be further expanded? Let us consider adding another $K_2$ to the last cluster. In this scenario, the resulting structure can be described as follows and illustrated in Figure \ref{fig:tree_3_k_K+2}:
\begin{itemize}
    \item the tree has one center, called $v_0$ that connect to three vertices $v_1, v_2, v_3$ 
    \item $v_1$ connect to $K_2 \cup K_2 \cup K_2$
    \item $v_2$ connect to $K_2 \cup \dots \cup K_2$, $k$ times
    \item $v_3$ also connect to another $K_2 \cup \dots \cup K_2$, $k+2$ times
\end{itemize}

\begin{figure}[H]
    \centering
    \begin{tikzpicture}[scale=0.7]
        \tikzstyle{black}=[fill=black, draw=black, shape=circle]
        
		\node [style=black] (0) at (0, 0) {};
		\node [style=black] (1) at (1, 0) {};
		\node [style=black] (2) at (2, 0) {};
		\node [style=black] (3) at (0, 1) {};
		\node [style=black] (4) at (1, 1) {};
		\node [style=black] (5) at (2, 1) {};
		\node [style=black] (6) at (1, 3) {};
		\node [style=black] (7) at (4, 0) {};
		\node [style=black] (8) at (5, 0) {};
		\node [style=black] (9) at (7, 0) {};
		\node [style=black] (10) at (4, 1) {};
		\node [style=black] (11) at (5, 1) {};
		\node [style=black] (12) at (7, 1) {};
		\node [style=black] (13) at (5.5, 3) {};
		\node [style=black] (14) at (9, 0) {};
		\node [style=black] (15) at (10, 0) {};
		\node [style=black] (16) at (12, 0) {};
		\node [style=black] (17) at (9, 1) {};
		\node [style=black] (18) at (10, 1) {};
		\node [style=black] (19) at (12, 1) {};
		\node [style=black] (20) at (10.5, 3) {};
		\node [style=black] (21) at (5.5, 5) {};
		
		\draw node[fill,circle, inner sep=1] at (5.7,0.5) {};
		\draw node[fill,circle, inner sep=1] at (6,0.5) {};
		\draw node[fill,circle, inner sep=1] at (6.3,0.5) {};
		\draw node[fill,circle, inner sep=1] at (10.7,0.5) {};
		\draw node[fill,circle, inner sep=1] at (11,0.5) {};
		\draw node[fill,circle, inner sep=1] at (11.3,0.5) {};
		
        \draw [decorate,decoration={brace,amplitude=5pt,mirror,raise=4ex}] (4,0) -- (7,0) node[midway,yshift=-3em]{$k$ times};
        
        \draw [decorate,decoration={brace,amplitude=5pt,mirror,raise=4ex}] (9,0) -- (12,0) node[midway,yshift=-3em]{$k+2$ times};
		\draw (3) to (0);
		\draw (4) to (1);
		\draw (5) to (2);
		\draw (6) to (5);
		\draw (6) to (4);
		\draw (6) to (3);
		\draw (10) to (7);
		\draw (11) to (8);
		\draw (12) to (9);
		\draw (13) to (12);
		\draw (13) to (11);
		\draw (13) to (10);
		\draw (20) to (17);
		\draw (20) to (18);
		\draw (20) to (19);
		\draw (17) to (14);
		\draw (18) to (15);
		\draw (19) to (16);
		\draw (21) to (20);
		\draw (21) to (13);
		\draw (21) to (6);
		\node at (5.5,5.5) {$v_0$};
		\node at (0.3,3) {$v_1$};
		\node at (4.8,3) {$v_2$};
		\node at (9.8,3) {$v_3$};
\end{tikzpicture}

    \caption{Illustration of 3,k,k+2 structure of trees that have non-log-concave independence polynomials. }
    \label{fig:tree_3_k_K+2}
\end{figure}
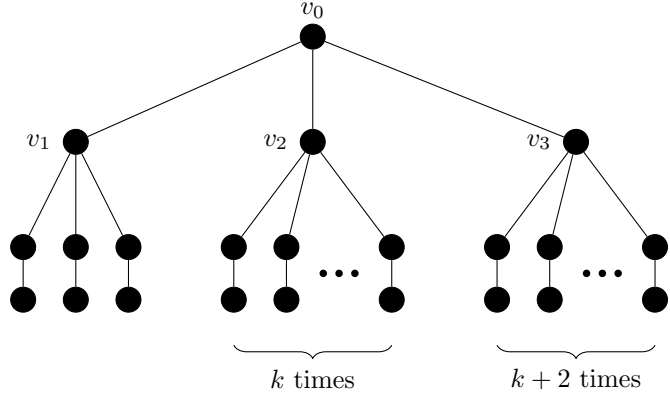

\begin{theorem}
    All trees from $3,k,k+2$ structure, where $k \geq 4$, have non-log-concave independence polynomials. 
\end{theorem}

\begin{proof}
Let compute the independence polynomial $3,k,k+2$-structure, and choose $v_0$ to be the first vertex to remove from the graph.
\[I(G)_{v_0} = I(G-v_0)+x \cdot I(G-N[v_0])\]
Expand the first term in that sum:
\[
\begin{split}
I(G-v_0) &= \!\begin{multlined}[t]
[(2x+1)^{k+2}+x (x+1)^{k+2}] \cdot [(2x+1)^k+x (x+1)^k] \\
\cdot [(2x+1)^3+x (x+1)^3]
\end{multlined}\\
&= \!\begin{multlined}[t]
\bigg[\sum_{i=0}^{k+2} \binom{k+2}{i} (2x)^{i} +x\sum_{i=0}^{k+2} \binom{k+2}{i} x^{i}\bigg] \cdot \\ 
\bigg[\sum_{i=0}^{k} \binom{k}{i} (2x)^{i} +x\sum_{i=0}^{k} \binom{k}{i} x^{i}\bigg] \cdot \\
[(2x+1)^3+x (x+1)^3]
\end{multlined}\\
&= \!\begin{multlined}[t]
[x^{k+3}+(2^{k+2}+k+2)x^{k+2}+\\((k+2)\cdot2^{k+1}+\frac{(k+1)(k+2)}{2})x^{k+1}+\dots] \cdot \\
[x^{k+1}+(2^k+k)x^k+(k\cdot2^{k-1}+\frac{k(k-1)}{2})x^{k-1}+\dots] \cdot\\
 [x^4+11x^3+15x^2+ \dots]
\end{multlined}\\
\end{split}
\]

We can divide the first term expression into three factors $A \cdot B \cdot C$ such that:
\begin{itemize}
    \item $A = [x^{k+3}+(2^{k+2}+k+2)x^{k+2}+((k+2)\cdot2^{k+1}+\frac{(k+1)(k+2)}{2})x^{k+1}+\dots]$
    \item $B = [x^{k+1}+(2^k+k)x^k+(k\cdot2^{k-1}+\frac{k(k-1)}{2})x^{k-1}+\dots]$
    \item $C = [x^4+11x^3+15x^2+ \dots]$
\end{itemize}

Expand the second term in that sum:
\[
\begin{split}
I(G-N[v_0]) &= x \cdot (2x+1)^{2k+5}\\
&= x \cdot \bigg[\sum_{i=0}^{2k+5} \binom{2k+5}{i} (2x)^{i}\bigg]\\
&= x \cdot [(2k)^{2k+5}+\dots]\\
&= 2^{2k+5}x^{2k+6}+\dots
\end{split}
\]

The highest exponent that can reach is $2k+8$, from taking the biggest exponent from every factor, from factor $A$ we multiply by $x^{k+3}$, from factor $B$ we multiply by $x^{k+1}$ and from $C$ factor we multiply by $x^4$, so 
\[
x^{k+3} \cdot x^{k+1} \cdot x^4 = x^{2k+8},
\]
where the final coefficient is equal to $1$.\\

To calculate the coefficient of $X^{2k+7}$ we have 3 options:
\begin{itemize}
    \item multiply $x^{k+3}$ from $A$, $x^{k+1}$ from $B$ and $11x^3$ from $C$
    \item multiply $x^{k+3}$ from $A$, $(2^k+k)x^k$ from $B$ and $x^4$ from $C$
    \item multiply $(2^{k+2}+k+2)x^{k+2}$ from $A$, $x^{k+1}$ from $B$ and $x^4$ from $C$
\end{itemize}

\[
\!\begin{multlined}[t]
x^{k+3} \cdot x^{k+1} \cdot 11x^3 + x^{k+3} \cdot (2^k+k)x^k \cdot x^4 + (2^{k+2}+k+2)x^{k+2} \cdot x^{k+1} \cdot x^4 =\\
( 5 \cdot2^k +2k + 13)x^{2k + 7},
\end{multlined}
\]

where the final coefficient is equal to $5 \cdot2^k +2k + 13$.\\

To calculate the coefficient of $x^{2k+6}$ we have 6 options:

\begin{itemize}
    \item multiply $x^{k+3}$ from $A$, $x^{k+1}$ from $B$ and $15x^2$ from $C$
    \item multiply $x^{k+3}$ from $A$, $(2^k+k)x^k$ from $B$ and $11x^3$ from $C$
    \item multiply $x^{k+3}$ from $A$, $(k\cdot2^{k-1}+\frac{k(k-1)}{2})x^{k-1}$ from $B$ and $x^4$ from $C$
    \item multiply $(2^{k+2}+k+2)x^{k+2}$ from $A$, $x^{k+1}$ from $B$ and $11x^3$ from $C$
    \item multiply $(2^{k+2}+k+2)x^{k+2}$ from $A$, $(2^k+k)x^k$ from $B$ and $x^4$ from $C$
    \item multiply $((k+2)\cdot2^{k+1}+\frac{(k+1)(k+2)}{2})x^{k+1}$ from $A$, $x^{k+1}$ from $B$ and $x^4$ from $C$
\end{itemize} 
and also we add the coefficient $2^{2k+5}$ from the second term.

\[
\begin{multlined}
    x^{k+3} \cdot x^{k+1} \cdot 15x^2 + 
x^{k+3} \cdot (2^k+k)x^k \cdot 11x^3 + \\
x^{k+3} \cdot (k\cdot2^{k-1}+\frac{k(k-1)}{2})x^{k-1} \cdot x^4 + 
(2^{k+2}+k+2)x^{k+2} \cdot x^{k+1} \cdot 11x^3 + \\
(2^{k+2}+k+2)x^{k+2} \cdot (2^k+k)x^k \cdot x^4 + \\
((k+2)\cdot2^{k+1}+\frac{(k+1)(k+2)}{2})x^{k+1} \cdot x^{k+1} \cdot x^4 + 2^{2k+5}x^{2k+6}\\=
\frac{1}{2} (2\cdot(61\cdot2^k + 9\cdot4^{k + 1} + 38) + k (4k + 15\cdot2^k + 50)) x^{2 k + 6},
\end{multlined}
\]

where the final coefficient is equal to $\frac{1}{2} (2\cdot(61\cdot2^k + 9\cdot4^{k + 1} + 38) + k (4k + 15\cdot2^k + 50))$.\\

So the independence polynomial is:

\begin{multlined}[t]
I(G) = x^{2k+8} + (5 \cdot2^k +2k + 13)x^{2k+7} + [\frac{1}{2} (2\cdot(61\cdot2^k + 9\cdot4^{k + 1} + 38)\\ + k (4k + 15\cdot2^k + 50))]x^{2k+6}+\dots
\end{multlined}

Now, lets prove the non-log-concave of the $x^{2k+6}$ term in the independence polynomial:
\[
(5 \cdot2^k +2k + 13)^2 < 1 \cdot [\frac{1}{2} (2\cdot(61\cdot2^k + 9\cdot4^{k + 1} + 38) + k (4k + 15\cdot2^k + 50))]
\]
The left-hand and right-hand sides equal $k \approx 3.61719$ so the left-hand side is smaller than the right-hand side from $k=4$ and above.
\end{proof}

\section{Extensions of T2}
\subsection{3*,k,k+1 structure}
Similar to $T_1$, $T_2$ in \textbf{Figure \ref{fig:two_trees}} is also just an instance of an infinite family of trees whose independence polynomials do not exhibit log-concavity. The left subtree of $T_2$ can be denoted as "3*". The structure of these trees, which we refer to as the \textit{3*,k,k+1 structure}, is described below and illustrated in Figure \ref{fig:tree_3*_k_K+1}:
\begin{itemize}
    \item the tree has one center, denoted $v_0$ that is connected to three vertices $v_1, v_2, v_3$ 
    \item $v_1$ is connected to $P_4 \cup K_2 \cup K_2$
    \item $v_2$ is connected to $K_2 \cup \dots \cup K_2$, $k$ times
    \item $v_3$ is connected to another $K_2 \cup \dots \cup K_2$, $k+1$ times
\end{itemize}

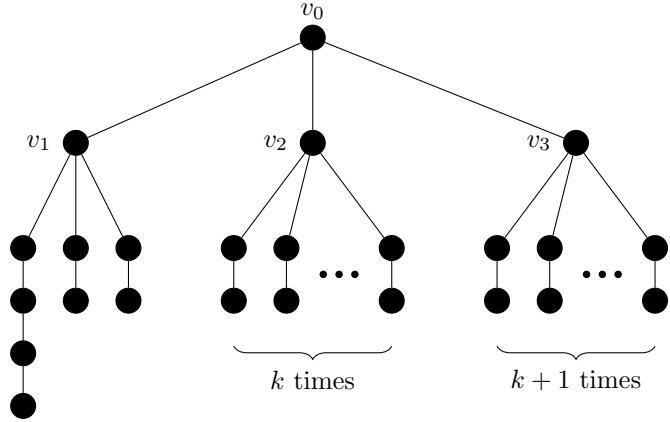
\begin{figure}[H]
    \centering
    \begin{tikzpicture}[scale=0.7]
        \tikzstyle{black}=[fill=black, draw=black, shape=circle]
        
		\node [style=black] (0) at (0, 0) {};
		\node [style=black] (1) at (1, 0) {};
		\node [style=black] (2) at (2, 0) {};
		\node [style=black] (3) at (0, 1) {};
		\node [style=black] (4) at (1, 1) {};
		\node [style=black] (5) at (2, 1) {};
		\node [style=black] (6) at (1, 3) {};
		\node [style=black] (7) at (4, 0) {};
		\node [style=black] (8) at (5, 0) {};
		\node [style=black] (9) at (7, 0) {};
		\node [style=black] (10) at (4, 1) {};
		\node [style=black] (11) at (5, 1) {};
		\node [style=black] (12) at (7, 1) {};
		\node [style=black] (13) at (5.5, 3) {};
		\node [style=black] (14) at (9, 0) {};
		\node [style=black] (15) at (10, 0) {};
		\node [style=black] (16) at (12, 0) {};
		\node [style=black] (17) at (9, 1) {};
		\node [style=black] (18) at (10, 1) {};
		\node [style=black] (19) at (12, 1) {};
		\node [style=black] (20) at (10.5, 3) {};
		\node [style=black] (21) at (5.5, 5) {};
		\node [style=black] (22) at (0, -1) {};
		\node [style=black] (23) at (0, -2) {};
		\draw node[fill,circle, inner sep=1] at (5.7,0.5) {};
		\draw node[fill,circle, inner sep=1] at (6,0.5) {};
		\draw node[fill,circle, inner sep=1] at (6.3,0.5) {};
		\draw node[fill,circle, inner sep=1] at (10.7,0.5) {};
		\draw node[fill,circle, inner sep=1] at (11,0.5) {};
		\draw node[fill,circle, inner sep=1] at (11.3,0.5) {};
		
        \draw [decorate,decoration={brace,amplitude=5pt,mirror,raise=4ex}] (4,0) -- (7,0) node[midway,yshift=-3em]{$k$ times};
        
        \draw [decorate,decoration={brace,amplitude=5pt,mirror,raise=4ex}] (9,0) -- (12,0) node[midway,yshift=-3em]{$k+1$ times};
		\draw (3) to (0);
		\draw (4) to (1);
		\draw (5) to (2);
		\draw (6) to (5);
		\draw (6) to (4);
		\draw (6) to (3);
		\draw (10) to (7);
		\draw (11) to (8);
		\draw (12) to (9);
		\draw (13) to (12);
		\draw (13) to (11);
		\draw (13) to (10);
		\draw (20) to (17);
		\draw (20) to (18);
		\draw (20) to (19);
		\draw (17) to (14);
		\draw (18) to (15);
		\draw (19) to (16);
		\draw (21) to (20);
		\draw (21) to (13);
		\draw (21) to (6);
		\draw (22) to (0);
		\draw (23) to (22);
		\node at (5.5,5.5) {$v_0$};
		\node at (0.3,3) {$v_1$};
		\node at (4.8,3) {$v_2$};
		\node at (9.8,3) {$v_3$};
\end{tikzpicture}

    \caption{An illustration of the 3*,k,k+1 structure of trees that have non-log-concave independence polynomials. }
    \label{fig:tree_3*_k_K+1}
\end{figure}

\begin{lemma}[\cite{KadrawiLevitYosefMizrachi2023}]
    All trees of the $3^*,k,k+1$ structure, where $k \geq 3$, have non-log-concave independence polynomials.  
\end{lemma}

\subsection{3*,k,k+2 structure}
Following the success of extending $T_1$, we attempted to add another $K_2$ to the left sub-tree of $T_2$. The resulting structure is denoted as the \textit{3*, k, k+2 structure}, and is described below. The structure is also depicted in Figure \ref{fig:tree_3*_k_K+2}.
\begin{itemize}
    \item the tree has one center, denoted $v_0$ that is connected to three vertices $v_1, v_2, v_3$ 
    \item $v_1$ is connected to $P_4 \cup K_2 \cup K_2$
    \item $v_2$ is connected to $K_2 \cup \dots \cup K_2$, $k$ times
    \item $v_3$ is connected to another $K_2 \cup \dots \cup K_2$, $k+2$ times
\end{itemize}

\begin{figure}[H]
    \centering
    \begin{tikzpicture}[scale=0.7]
        \tikzstyle{black}=[fill=black, draw=black, shape=circle]
        
		\node [style=black] (0) at (0, 0) {};
		\node [style=black] (1) at (1, 0) {};
		\node [style=black] (2) at (2, 0) {};
		\node [style=black] (3) at (0, 1) {};
		\node [style=black] (4) at (1, 1) {};
		\node [style=black] (5) at (2, 1) {};
		\node [style=black] (6) at (1, 3) {};
		\node [style=black] (7) at (4, 0) {};
		\node [style=black] (8) at (5, 0) {};
		\node [style=black] (9) at (7, 0) {};
		\node [style=black] (10) at (4, 1) {};
		\node [style=black] (11) at (5, 1) {};
		\node [style=black] (12) at (7, 1) {};
		\node [style=black] (13) at (5.5, 3) {};
		\node [style=black] (14) at (9, 0) {};
		\node [style=black] (15) at (10, 0) {};
		\node [style=black] (16) at (12, 0) {};
		\node [style=black] (17) at (9, 1) {};
		\node [style=black] (18) at (10, 1) {};
		\node [style=black] (19) at (12, 1) {};
		\node [style=black] (20) at (10.5, 3) {};
		\node [style=black] (21) at (5.5, 5) {};
		\node [style=black] (22) at (0, -1) {};
		\node [style=black] (23) at (0, -2) {};
		\draw node[fill,circle, inner sep=1] at (5.7,0.5) {};
		\draw node[fill,circle, inner sep=1] at (6,0.5) {};
		\draw node[fill,circle, inner sep=1] at (6.3,0.5) {};
		\draw node[fill,circle, inner sep=1] at (10.7,0.5) {};
		\draw node[fill,circle, inner sep=1] at (11,0.5) {};
		\draw node[fill,circle, inner sep=1] at (11.3,0.5) {};
		
        \draw [decorate,decoration={brace,amplitude=5pt,mirror,raise=4ex}] (4,0) -- (7,0) node[midway,yshift=-3em]{$k$ times};
        
        \draw [decorate,decoration={brace,amplitude=5pt,mirror,raise=4ex}] (9,0) -- (12,0) node[midway,yshift=-3em]{$k+2$ times};
		\draw (3) to (0);
		\draw (4) to (1);
		\draw (5) to (2);
		\draw (6) to (5);
		\draw (6) to (4);
		\draw (6) to (3);
		\draw (10) to (7);
		\draw (11) to (8);
		\draw (12) to (9);
		\draw (13) to (12);
		\draw (13) to (11);
		\draw (13) to (10);
		\draw (20) to (17);
		\draw (20) to (18);
		\draw (20) to (19);
		\draw (17) to (14);
		\draw (18) to (15);
		\draw (19) to (16);
		\draw (21) to (20);
		\draw (21) to (13);
		\draw (21) to (6);
		\draw (22) to (0);
		\draw (23) to (22);
		\node at (5.5,5.5) {$v_0$};
		\node at (0.3,3) {$v_1$};
		\node at (4.8,3) {$v_2$};
		\node at (9.8,3) {$v_3$};
\end{tikzpicture}

    \caption{An illustration of the 3*,k,k+2 structure of trees that have non-log-concave independence polynomials. }
    \label{fig:tree_3*_k_K+2}
\end{figure}
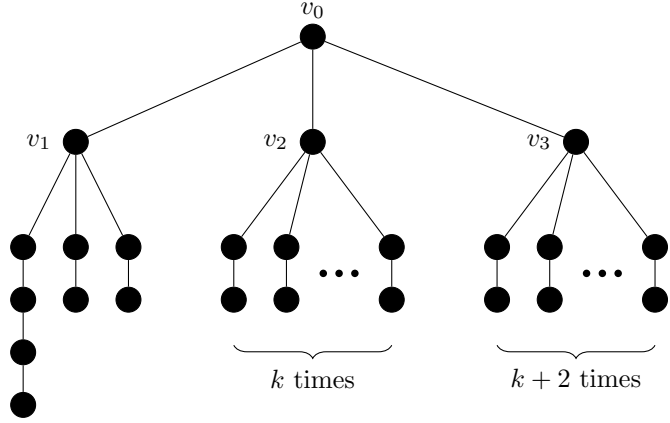

\begin{lemma}
    All trees of the $3^*,k,k+2$ structure, where $k \geq 3$, have non-log-concave independence polynomials.  
\end{lemma}

\begin{proof}
Let us compute the independence polynomial of a tree having $3^*,k,k+2$ structure, and choose $v_0$ to be the first vertex to remove from the graph.
\[I(G)_{v_0} = I(G-v_0)+x \cdot I(G-N[v_0])\]
Expand the first term in that sum:
\[
\begin{split}
I(G-v_0)&= \!\begin{multlined}[t]
[(2x+1)^{k+2}+x(x+1)^{k+2}] \cdot [(2x+1)^k+x(x+1)^k] \cdot \\
[(2x+1)^2(3x^2+4x+1)+x(x+1)^2(x^2+3x+1)]
\end{multlined}\\
&= \!\begin{multlined}[t]
\bigg[\sum_{i=0}^{k+2} \binom{k+2}{i} (2x)^{i} +x\sum_{i=0}^{k+2} \binom{k+2}{i} x^{i}\bigg] 
\cdot \\ \bigg[\sum_{i=0}^{k} \binom{k}{i} (2x)^{i} +x\sum_{i=0}^{k} \binom{k}{i} x^{i}\bigg] \cdot \\
 [(2x+1)^2(3x^2+4x+1)+x(x+1)^2(x^2+3x+1)]
\end{multlined}\\
&= \!\begin{multlined}[t]
[x^{k+3}+(2^{k+2}+k+2)x^{k+2}+((k+2)\cdot2^{k+1}+\\\frac{(k+1)(k+2)}{2})x^{k+1}+\dots] \cdot \\
[x^{k+1}+(2^k+k)x^k+(k\cdot2^{k-1}+\frac{k(k-1)}{2})x^{k-1}+\dots] \cdot \\
[x^5 + 17 x^4 + 36 x^3 + \dots]
\end{multlined}\\
\end{split}
\]

We can divide the first term expression into three factors $A \cdot B \cdot C$ such that:
\begin{itemize}
    \item $A = [x^{k+3}+(2^{k+2}+k+2)x^{k+2}+((k+2)\cdot2^{k+1}+\frac{(k+1)(k+2)}{2})x^{k+1}+\dots]$
    \item $B = [x^{k+1}+(2^k+k)x^k+(k\cdot2^{k-1}+\frac{k(k-1)}{2})x^{k-1}+\dots]$
    \item $C = [x^5 + 17 x^4 + 36 x^3 + \dots]$
\end{itemize}

Expand the second term in that sum:
\[
\begin{split}
x\cdot I(G-N[v_0]) &= x \cdot (2x+1)^{2k+4}\cdot(3x^2+4x+1)\\
&= x(3x^2+4x+1) \cdot \bigg[\sum_{i=0}^{2k+4} \binom{2k+4}{i} (2x)^{i}]\\
&= x(3x^2+4x+1) \cdot [(2x)^{2k+4}+\dots]\\
&= 3\cdot2^{2 k + 4}x^{2 k + 7}+\dots
\end{split}
\]

The highest exponent that one can reach is $2k+9$. We obtain it by taking the highest exponent from every factor, i.e., from factor $A$ we choose $x^{k+3}$, from factor $B$ we choose $x^{k+1}$, while from $C$ factor we choose $x^5$, so 
\[
x^{n+3} \cdot x^{n+1} \cdot x^5 = x^{2k+9},
\]
where the final coefficient is equal to $1$.\\

To calculate the coefficient of $x^{2k+8}$ we have 3 following options:
\begin{itemize}
    \item multiply $x^{k+3}$ from $A$, $x^{k+1}$ from $B$ and $17 x^4$ from $C$
    \item multiply $x^{k+3}$ from $A$, $(2^k+k)x^k$ from $B$ and $x^5$ from $C$
    \item multiply $(2^{k+2}+k+2)x^{k+2}$ from $A$, $x^{k+1}$ from $B$ and $x^5$ from $C$
\end{itemize}

\[
\begin{multlined}
x^{k+3} \cdot x^{k+1} \cdot 17 x^4+
x^{k+3} \cdot (2^k+k)x^k \cdot x^5+
(2^{k+2}+k+2)x^{k+2} \cdot x^{k+1} \cdot x^5
 \\= (2 k + 2^k + 2^{k + 2} + 19) x^{2 k + 8},
\end{multlined}
\]

where the final coefficient is equal to $2 k + 2^k + 2^{k + 2} + 19$.\\

To calculate the coefficient of $x^{2k+7}$ we have 6 following options:
\begin{itemize}
    \item multiply $x^{k+3}$ from $A$, $x^{k+1}$ from $B$ and $36 x^3$ from $C$
    \item multiply $x^{k+3}$ from $A$, $(2^k+k)x^k$ from $B$ and $17 x^4$ from $C$
    \item multiply $x^{k+3}$ from $A$, $(k\cdot2^{k-1}+\frac{k(k-1)}{2})x^{k-1}$ from $B$ and $x^5$ from $C$
    \item multiply $(2^{k+2}+k+2)x^{k+2}$ from $A$, $x^{k+1}$ from $B$ and $17 x^4$ from $C$
    \item multiply $(2^{k+2}+k+2)x^{k+2}$ from $A$, $(2^k+k)x^k$ from $B$ and $x^5$ from $C$
    \item multiply $((k+2)\cdot2^{k+1}+\frac{(k+1)(k+2)}{2})x^{k+1}$ from $A$, $x^{k+1}$ from $B$ and $x^5$ from $C$
\end{itemize} 
and also we add the coefficient $3\cdot2^{2 k + 4}$ from the second term.

\[
\begin{multlined}[t]
x^{k+3} \cdot x^{k+1} \cdot 36 x^3 + 
x^{k+3} \cdot (2^k+k)x^k \cdot 17 x^4 + \\
x^{k+3} \cdot (k\cdot2^{k-1}+\frac{k(k-1)}{2})x^{k-1} \cdot x^5
+(2^{k+2}+k+2)x^{k+2} \cdot x^{k+1} \cdot 17 x^4 + \\
(2^{k+2}+k+2)x^{k+2} \cdot (2^k+k)x^k \cdot x^5 \\
((k+2)\cdot2^{k+1}+\frac{(k+1)(k+2)}{2})x^{k+1} \cdot x^{k+1} \cdot x^5
+ 3\cdot2^{2 k + 4}x^{2 k + 7} \\
=\frac{1}{2} (4\cdot k^2 + 15k \cdot 2^k+ 74k + 91\cdot2^{k + 1} + 13\cdot2^{2 k + 3} + 142) x^{2 k + 7},
\end{multlined}
\]

where the final coefficient is equal to $\frac{1}{2} (4\cdot k^2 + 15k \cdot 2^k+ 74k + 91\cdot2^{k + 1} + 13\cdot2^{2 k + 3} + 142)$.\\

Thus the independence polynomial is:

\begin{multlined}[t]
    I(G) = x^{2k+9} + (2k + 2^k + 2^{k + 2} + 19) x^{2k + 8} +\\
    \frac{1}{2} (4\cdot k^2 + 15k \cdot 2^k+ 74k + 91\cdot2^{k + 1} + 13\cdot2^{2 k + 3} + 142) x^{2 k + 7}+\dots
\end{multlined}

Now, let us prove the non-log-concavity of the $x^{2k+8}$ term in the independence polynomial:
\[
(2k + 2^k + 2^{k + 2} + 19)^2 < 1 \cdot \frac{1}{2} (4\cdot k^2 + 15k \cdot 2^k+ 74k + 91\cdot2^{k + 1} + 13\cdot2^{2 k + 3} + 142)
\]
The left-hand and right-hand sides are equal to $k \approx 2.83611$ so the left-hand side is smaller than the right-hand side from $k=3$ and above.

\end{proof}

\subsection{3*,k,k+3 structure}
To explore further extensions, we attempted to add another $K_2$ to the 3*,k,k+2 structure. The resulting structure is described below and depicted in Figure \ref{fig:tree_3*_k_K+3}:
\begin{itemize}
    \item the tree has one center, denoted $v_0$ that is connected to three vertices $v_1, v_2, v_3$ 
    \item $v_1$ is connected to $P_4 \cup K_2 \cup K_2$
    \item $v_2$ is connected to $K_2 \cup \dots \cup K_2$, $k$ times
    \item $v_3$ is connected to another $K_2 \cup \dots \cup K_2$, $k+3$ times
\end{itemize}

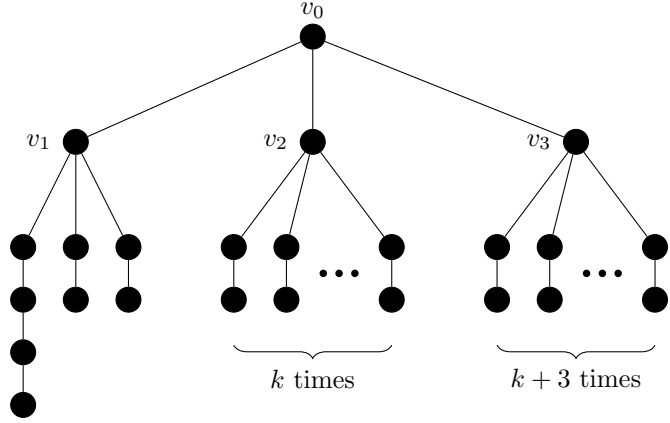
\begin{figure}[H]
    \centering
    \begin{tikzpicture}[scale=0.7]
        \tikzstyle{black}=[fill=black, draw=black, shape=circle]
        
		\node [style=black] (0) at (0, 0) {};
		\node [style=black] (1) at (1, 0) {};
		\node [style=black] (2) at (2, 0) {};
		\node [style=black] (3) at (0, 1) {};
		\node [style=black] (4) at (1, 1) {};
		\node [style=black] (5) at (2, 1) {};
		\node [style=black] (6) at (1, 3) {};
		\node [style=black] (7) at (4, 0) {};
		\node [style=black] (8) at (5, 0) {};
		\node [style=black] (9) at (7, 0) {};
		\node [style=black] (10) at (4, 1) {};
		\node [style=black] (11) at (5, 1) {};
		\node [style=black] (12) at (7, 1) {};
		\node [style=black] (13) at (5.5, 3) {};
		\node [style=black] (14) at (9, 0) {};
		\node [style=black] (15) at (10, 0) {};
		\node [style=black] (16) at (12, 0) {};
		\node [style=black] (17) at (9, 1) {};
		\node [style=black] (18) at (10, 1) {};
		\node [style=black] (19) at (12, 1) {};
		\node [style=black] (20) at (10.5, 3) {};
		\node [style=black] (21) at (5.5, 5) {};
		\node [style=black] (22) at (0, -1) {};
		\node [style=black] (23) at (0, -2) {};
		\draw node[fill,circle, inner sep=1] at (5.7,0.5) {};
		\draw node[fill,circle, inner sep=1] at (6,0.5) {};
		\draw node[fill,circle, inner sep=1] at (6.3,0.5) {};
		\draw node[fill,circle, inner sep=1] at (10.7,0.5) {};
		\draw node[fill,circle, inner sep=1] at (11,0.5) {};
		\draw node[fill,circle, inner sep=1] at (11.3,0.5) {};
		
        \draw [decorate,decoration={brace,amplitude=5pt,mirror,raise=4ex}] (4,0) -- (7,0) node[midway,yshift=-3em]{$k$ times};
        
        \draw [decorate,decoration={brace,amplitude=5pt,mirror,raise=4ex}] (9,0) -- (12,0) node[midway,yshift=-3em]{$k+3$ times};
		\draw (3) to (0);
		\draw (4) to (1);
		\draw (5) to (2);
		\draw (6) to (5);
		\draw (6) to (4);
		\draw (6) to (3);
		\draw (10) to (7);
		\draw (11) to (8);
		\draw (12) to (9);
		\draw (13) to (12);
		\draw (13) to (11);
		\draw (13) to (10);
		\draw (20) to (17);
		\draw (20) to (18);
		\draw (20) to (19);
		\draw (17) to (14);
		\draw (18) to (15);
		\draw (19) to (16);
		\draw (21) to (20);
		\draw (21) to (13);
		\draw (21) to (6);
		\draw (22) to (0);
		\draw (23) to (22);
		\node at (5.5,5.5) {$v_0$};
		\node at (0.3,3) {$v_1$};
		\node at (4.8,3) {$v_2$};
		\node at (9.8,3) {$v_3$};
\end{tikzpicture}

    \caption{An illustration of the 3*,k,k+3 structure of trees that have non-log-concave independence polynomials. }
    \label{fig:tree_3*_k_K+3}
\end{figure}

\begin{lemma}
    All trees of the $3^*,k,k+3$ structure, where $k \geq 4$, have non-log-concave independence polynomials.  
\end{lemma}

\begin{proof}
Let us compute the independence polynomial of a tree having $3^*,k,k+3$ structure, and choose $v_0$ to be the first vertex to remove from the graph.
\[I(G)_{v_0} = I(G-v_0)+x \cdot I(G-N[v_0])\]
Expand the first term in that sum:
\[
\begin{split}
I(G-v_0) &= \!\begin{multlined}[t]
[(2x+1)^{k+3}+x(x+1)^{k+3}] \cdot [(2x+1)^k+x(x+1)^k] \cdot \\
[(2x+1)^2(3x^2+4x+1)+x(x+1)^2(x^2+3x+1)]
\end{multlined}\\
&= \!\begin{multlined}[t]
\bigg[\sum_{i=0}^{k+3} \binom{k+3}{i} (2x)^{i} +x\sum_{i=0}^{k+3} \binom{k+3}{i} x^{i}\bigg] 
\cdot \\ \bigg[\sum_{i=0}^{k} \binom{k}{i} (2x)^{i} +x\sum_{i=0}^{k} \binom{k}{i} x^{i}\bigg] \cdot \\
[(2x+1)^2(3x^2+4x+1)+x(x+1)^2(x^2+3x+1)]
\end{multlined}\\
&= \!\begin{multlined}[t]
[x^{k+4}+(2^{k+3}+k+3)x^{k+3}+((k+3)\cdot2^{k+2}+\\\frac{(k+2)(k+3)}{2})x^{k+2}+\dots]
\cdot \\ [x^{k+1}+(2^k+k)x^k+(k\cdot2^{k-1}+\frac{k(k-1)}{2})x^{k-1}+\dots] \cdot \\
[x^5 + 17 x^4 + 36 x^3 + \dots]
\end{multlined}\\
\end{split}
\]

We can divide the first term expression into three factors $A \cdot B \cdot C$ such that:
\begin{itemize}
    \item $A = [x^{k+4}+(2^{k+3}+k+3)x^{k+3}+((k+3)\cdot2^{k+2}+\frac{(k+2)(k+3)}{2})x^{k+2}+\dots]$
    \item $B = [x^{k+1}+(2^k+k)x^k+(k\cdot2^{k-1}+\frac{k(k-1)}{2})x^{k-1}+\dots]$
    \item $C = [x^5 + 17 x^4 + 36 x^3 + \dots]$
\end{itemize}

Expand the second term in that sum:
\[
\begin{split}
x\cdot I(G-N[v_0]) &= x \cdot (2x+1)^{2k+4}\cdot(3x^2+4x+1)\\
&= x(3x^2+4x+1) \cdot \bigg[\sum_{i=0}^{2k+4} \binom{2k+4}{i} (2x)^{i}]\\
&= x(3x^2+4x+1) \cdot [(2x)^{2k+4}+\dots]\\
&= 3\cdot2^{2 k + 5}x^{2 k + 8}+\dots
\end{split}
\]

The highest exponent that one can reach is $2k+10$. We obtain it by taking the highest exponent from every factor, i.e., from factor $A$ we choose $x^{k+4}$, from factor $B$ we choose $x^{k+1}$, while from $C$ factor we choose $x^5$, so 
\[
x^{n+4} \cdot x^{n+1} \cdot x^5 = x^{2k+10},
\]
where the final coefficient is equal to $1$.\\

To calculate the coefficient of $x^{2k+9}$ we have 3 following options:
\begin{itemize}
    \item multiply $x^{k+4}$ from $A$, $x^{k+1}$ from $B$ and $17 x^4$ from $C$
    \item multiply $x^{k+4}$ from $A$, $(2^k+k)x^k$ from $B$ and $x^5$ from $C$
    \item multiply $(2^{k+3}+k+3)x^{k+3}$ from $A$, $x^{k+1}$ from $B$ and $x^5$ from $C$
\end{itemize}

\[
\begin{multlined}
x^{k+4} \cdot x^{k+1} \cdot 17 x^4+
x^{k+4} \cdot (2^k+k)x^k \cdot x^5+
(2^{k+3}+k+3)x^{k+3} \cdot x^{k+1} \cdot x^5
 \\= (2 k + 9\cdot2^k + 20) x^{2 k + 9},
\end{multlined}
\]

where the final coefficient is equal to $2 k + 9\cdot2^k + 20$.\\

To calculate the coefficient of $x^{2k+8}$ we have 6 following options:
\begin{itemize}
    \item multiply $x^{k+4}$ from $A$, $x^{k+1}$ from $B$ and $36 x^3$ from $C$
    \item multiply $x^{k+4}$ from $A$, $(2^k+k)x^k$ from $B$ and $17 x^4$ from $C$
    \item multiply $x^{k+4}$ from $A$, $(k\cdot2^{k-1}+\frac{k(k-1)}{2})x^{k-1}$ from $B$ and $x^5$ from $C$
    \item multiply $(2^{k+3}+k+3)x^{k+3}$ from $A$, $x^{k+1}$ from $B$ and $17 x^4$ from $C$
    \item multiply $(2^{k+3}+k+3)x^{k+3}$ from $A$, $(2^k+k)x^k$ from $B$ and $x^5$ from $C$
    \item multiply $((k+3)\cdot2^{k+2}+\frac{(k+2)(k+3)}{2})x^{k+2}$ from $A$, $x^{k+1}$ from $B$ and $x^5$ from $C$
\end{itemize} 
and also we add the coefficient $3\cdot2^{2 k + 5}$ from the second term.

\[
\begin{multlined}
x^{k+4} \cdot x^{k+1} \cdot 36 x^3 + 
x^{k+4} \cdot (2^k+k)x^k \cdot 17 x^4 + \\
x^{k+4} \cdot (k\cdot2^{k-1}+\frac{k(k-1)}{2})x^{k-1} \cdot x^5 +
(2^{k+3}+k+3)x^{k+3} \cdot x^{k+1} \cdot 17 x^4 + \\
(2^{k+3}+k+3)x^{k+3} \cdot (2^k+k)x^k \cdot x^5 + \\
((k+3)\cdot2^{k+2}+\frac{(k+2)(k+3)}{2})x^{k+2} \cdot x^{k+1} \cdot x^5
+ 3\cdot2^{2 k + 5}x^{2 k +8} \\
=\frac{1}{2} (k (4 k + 27\cdot2^k + 78) + 13\cdot4^{k + 2} + 21 \cdot2^{k + 4} + 180) x^{2 k + 8},
\end{multlined}
\]

where the final coefficient is equal to $\frac{1}{2} (k (4 k + 27\cdot2^k + 78) + 13\cdot4^{k + 2} + 21 \cdot2^{k + 4} + 180)$.\\

Thus the independence polynomial is:

\begin{multlined}[t]
I(G) = x^{2k+10} + 
(2 k + 9\cdot2^k + 20) x^{2 k + 9} +\\
\frac{1}{2} (k (4 k + 27\cdot2^k + 78) + 13\cdot4^{k + 2} + 21 \cdot2^{k + 4} + 180) x^{2 k + 8} +
\dots
\end{multlined}

Now, let us prove the non-log-concavity of the $x^{2k+9}$ term in the independence polynomial:
\[
(2 k + 9\cdot2^k + 20)^2 < 1 \cdot \frac{1}{2} (k (4 k + 27\cdot2^k + 78) + 13\cdot4^{k + 2} + 21 \cdot2^{k + 4} + 180)
\]
The left-hand and right-hand sides are equal to $k \approx 3.76626$ so the left-hand side is smaller than the right-hand side from $k=4$ and above.

\end{proof}

\subsection{3*,k,k structure}

To further extend the structure, we also considered the base case where the last two clusters are equal. The resulting structure is described below and illustrated in Figure \ref{fig:tree_3*_k_K}:

\begin{itemize}
    \item the tree has one center, denoted $v_0$ that is connected to three vertices $v_1, v_2, v_3$ 
    \item $v_1$ is connected to $P_4 \cup K_2 \cup K_2$
    \item $v_2$ is connected to $K_2 \cup \dots \cup K_2$, $k$ times
    \item $v_3$ is connected to another $K_2 \cup \dots \cup K_2$, $k$ times
\end{itemize}

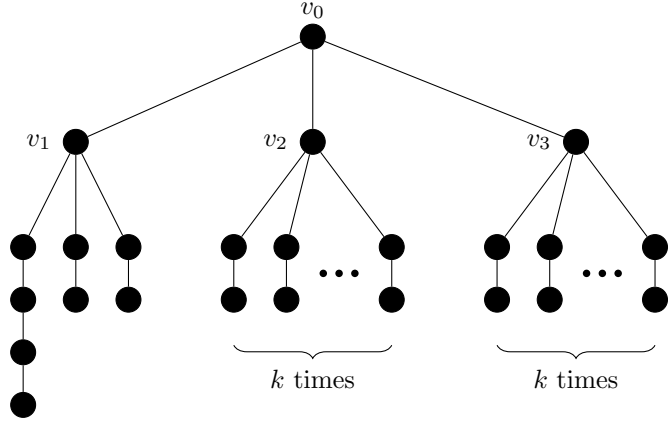
\begin{figure}[H]
    \centering
    \begin{tikzpicture}[scale=0.7]
        \tikzstyle{black}=[fill=black, draw=black, shape=circle]
        
		\node [style=black] (0) at (0, 0) {};
		\node [style=black] (1) at (1, 0) {};
		\node [style=black] (2) at (2, 0) {};
		\node [style=black] (3) at (0, 1) {};
		\node [style=black] (4) at (1, 1) {};
		\node [style=black] (5) at (2, 1) {};
		\node [style=black] (6) at (1, 3) {};
		\node [style=black] (7) at (4, 0) {};
		\node [style=black] (8) at (5, 0) {};
		\node [style=black] (9) at (7, 0) {};
		\node [style=black] (10) at (4, 1) {};
		\node [style=black] (11) at (5, 1) {};
		\node [style=black] (12) at (7, 1) {};
		\node [style=black] (13) at (5.5, 3) {};
		\node [style=black] (14) at (9, 0) {};
		\node [style=black] (15) at (10, 0) {};
		\node [style=black] (16) at (12, 0) {};
		\node [style=black] (17) at (9, 1) {};
		\node [style=black] (18) at (10, 1) {};
		\node [style=black] (19) at (12, 1) {};
		\node [style=black] (20) at (10.5, 3) {};
		\node [style=black] (21) at (5.5, 5) {};
		\node [style=black] (22) at (0, -1) {};
		\node [style=black] (23) at (0, -2) {};
		\draw node[fill,circle, inner sep=1] at (5.7,0.5) {};
		\draw node[fill,circle, inner sep=1] at (6,0.5) {};
		\draw node[fill,circle, inner sep=1] at (6.3,0.5) {};
		\draw node[fill,circle, inner sep=1] at (10.7,0.5) {};
		\draw node[fill,circle, inner sep=1] at (11,0.5) {};
		\draw node[fill,circle, inner sep=1] at (11.3,0.5) {};
		
        \draw [decorate,decoration={brace,amplitude=5pt,mirror,raise=4ex}] (4,0) -- (7,0) node[midway,yshift=-3em]{$k$ times};
        
        \draw [decorate,decoration={brace,amplitude=5pt,mirror,raise=4ex}] (9,0) -- (12,0) node[midway,yshift=-3em]{$k$ times};
		\draw (3) to (0);
		\draw (4) to (1);
		\draw (5) to (2);
		\draw (6) to (5);
		\draw (6) to (4);
		\draw (6) to (3);
		\draw (10) to (7);
		\draw (11) to (8);
		\draw (12) to (9);
		\draw (13) to (12);
		\draw (13) to (11);
		\draw (13) to (10);
		\draw (20) to (17);
		\draw (20) to (18);
		\draw (20) to (19);
		\draw (17) to (14);
		\draw (18) to (15);
		\draw (19) to (16);
		\draw (21) to (20);
		\draw (21) to (13);
		\draw (21) to (6);
		\draw (22) to (0);
		\draw (23) to (22);
		\node at (5.5,5.5) {$v_0$};
		\node at (0.3,3) {$v_1$};
		\node at (4.8,3) {$v_2$};
		\node at (9.8,3) {$v_3$};
\end{tikzpicture}

    \caption{An illustration of the 3*,k,k structure of trees that have non-log-concave independence polynomials. }
    \label{fig:tree_3*_k_K}
\end{figure}

\begin{lemma}
    All trees of the $3^*,k,k$ structure, where $k \geq 4$, have non-log-concave independence polynomials.  
\end{lemma}

\begin{proof}
Let us compute the independence polynomial of a tree having $3^*,k,k$ structure, and choose $v_0$ to be the first vertex to remove from the graph.
\[I(G)_{v_0} = I(G-v_0)+x \cdot I(G-N[v_0])\]
Expand the first term in that sum:
\[
\begin{split}
I(G-v_0) &= \!\begin{multlined}[t]
[(2x+1)^{k}+x(x+1)^{k}] \cdot [(2x+1)^k+x(x+1)^k] \cdot \\
[(2x+1)^2(3x^2+4x+1)+x(x+1)^2(x^2+3x+1)]
\end{multlined}\\
&= \!\begin{multlined}[t]
\bigg[\sum_{i=0}^{k} \binom{k}{i} (2x)^{i} +x\sum_{i=0}^{k} \binom{k}{i} x^{i}\bigg] 
\cdot \bigg[\sum_{i=0}^{k} \binom{k}{i} (2x)^{i} +x\sum_{i=0}^{k} \binom{k}{i} x^{i}\bigg] \cdot \\
[(2x+1)^2(3x^2+4x+1)+x(x+1)^2(x^2+3x+1)]
\end{multlined}\\
&= \!\begin{multlined}[t]
[x^{k+1}+(2^k+k)x^k+(k\cdot2^{k-1}+\frac{k(k-1)}{2})x^{k-1}+\dots] \cdot \\
[x^{k+1}+(2^k+k)x^k+(k\cdot2^{k-1}+\frac{k(k-1)}{2})x^{k-1}+\dots] \cdot \\
[x^5 + 17 x^4 + 36 x^3 + \dots]
\end{multlined}\\
\end{split}
\]

We can divide the first term expression into three factors $A \cdot B \cdot C$ such that:
\begin{itemize}
    \item $A = [x^{k+1}+(2^k+k)x^k+(k\cdot2^{k-1}+\frac{k(k-1)}{2})x^{k-1}+\dots]$
    \item $B = [x^{k+1}+(2^k+k)x^k+(k\cdot2^{k-1}+\frac{k(k-1)}{2})x^{k-1}+\dots]$
    \item $C = [x^5 + 17 x^4 + 36 x^3 + \dots]$
\end{itemize}

Expand the second term in that sum:
\[
\begin{split}
x\cdot I(G-N[v_0]) &= x \cdot (2x+1)^{2k+4}\cdot(3x^2+4x+1)\\
&= x(3x^2+4x+1) \cdot \bigg[\sum_{i=0}^{2k+4} \binom{2k+4}{i} (2x)^{i}]\\
&= x(3x^2+4x+1) \cdot [(2x)^{2k+1}+\dots]\\
&= 3\cdot2^{2 k + 2}x^{2 k + 5}+\dots
\end{split}
\]

The highest exponent that one can reach is $2k+7$. We obtain it by taking the highest exponent from every factor, i.e., from factor $A$ we choose $x^{k+1}$, from factor $B$ we choose $x^{k+1}$, while from $C$ factor we choose $x^5$, so 
\[
x^{n+1} \cdot x^{n+1} \cdot x^5 = x^{2k+7},
\]
where the final coefficient is equal to $1$.\\

To calculate the coefficient of $x^{2k+6}$ we have 3 following options:
\begin{itemize}
    \item multiply $x^{k+1}$ from $A$, $x^{k+1}$ from $B$ and $17 x^4$ from $C$
    \item multiply $x^{k+1}$ from $A$, $(2^k+k)x^k$ from $B$ and $x^5$ from $C$
    \item multiply $(2^k+k)x^{k}$ from $A$, $x^{k+1}$ from $B$ and $x^5$ from $C$
\end{itemize}

\[
\begin{multlined}
x^{k+1} \cdot x^{k+1} \cdot 17 x^4+
x^{k+1} \cdot (2^k+k)x^k \cdot x^5+
(2^{k}+k)x^{k} \cdot x^{k+1} \cdot x^5
 \\= (2 k + 2^{k + 1} + 17) x^{2 k + 6},
\end{multlined}
\]

where the final coefficient is equal to $2 k + 2^{k + 1} + 17$.\\

To calculate the coefficient of $x^{2k+5}$ we have 6 following options:
\begin{itemize}
    \item multiply $x^{k+1}$ from $A$, $x^{k+1}$ from $B$ and $36 x^3$ from $C$
    \item multiply $x^{k+1}$ from $A$, $(2^k+k)x^k$ from $B$ and $17 x^4$ from $C$
    \item multiply $x^{k+1}$ from $A$, $(k\cdot2^{k-1}+\frac{k(k-1)}{2})x^{k-1}$ from $B$ and $x^5$ from $C$
    \item multiply $(2^{k}+k)x^{k}$ from $A$, $x^{k+1}$ from $B$ and $17 x^4$ from $C$
    \item multiply $(2^{k}+k)x^{k}$ from $A$, $(2^k+k)x^k$ from $B$ and $x^5$ from $C$
    \item multiply $(k\cdot2^{k-1}+\frac{k(k-1)}{2})x^{k-1}$ from $A$, $x^{k+1}$ from $B$ and $x^5$ from $C$
\end{itemize} 
and also we add the coefficient $3\cdot2^{2 k + 2}$ from the second term.

\[
\begin{multlined}
x^{k+1} \cdot x^{k+1} \cdot 36 x^3 + 
x^{k+1} \cdot (2^k+k)x^k \cdot 17 x^4 + \\
x^{k+1} \cdot (k\cdot2^{k-1}+\frac{k(k-1)}{2})x^{k-1} \cdot x^5 +
(2^k+k)x^k \cdot x^{k+1} \cdot 17 x^4 + \\
(2^k+k)x^k \cdot (2^k+k)x^k \cdot x^5 + 
(k\cdot2^{k-1}+\frac{k(k-1)}{2})x^{k-1} \cdot x^{k+1} \cdot x^5 +\\
3\cdot2^{2 k + 2}x^{2 k +5}\\
=(2\cdot k^2 + 33 k + 13\cdot 4^k + 2^k\cdot(3 k + 34) + 36) x^{2 k + 5},
\end{multlined}
\]

where the final coefficient is equal to $2\cdot k^2 + 33 k + 13\cdot4^k + 2^k\cdot(3 k + 34) + 36$.\\

Thus the independence polynomial is:

\begin{multlined}[t]
I(G) = x^{2k+7} + 
(2 k + 2^{k + 1} + 17) x^{2 k + 6} +\\
(2 \cdot k^2 + 33 k + 13 \cdot 4^k + 2^k \cdot (3 k + 34) + 36) x^{2 k + 5} +
\dots
\end{multlined}

Now, let us prove the non-log-concavity of the $x^{2k+6}$ term in the independence polynomial:
\[
(2 k + 2^{k + 1} + 17)^2 < 1 \cdot (2 \cdot k^2 + 33 k + 13 \cdot 4^k + 2^k \cdot (3 k + 34) + 36)
\]
The left-hand and right-hand sides are equal to $k \approx 3.31871$ so the left-hand side is smaller than the right-hand side from $k=4$ and above.

\end{proof}

\section{Conclusion}

In this paper, we challenge the prevailing conjecture on the independence polynomials of trees, which suggests that they are unimodal and potentially log-concave. Instead, we make use of the two trees with 26 vertices whose independence polynomials are not log-concave, and introduce a number of infinite families of trees based on those two trees, whose independence polynomials are not log-concave as well.\\

Our findings suggest several exciting directions for future research, including the investigation of additional families of trees with structures of the form $3,k,k+j$ where $j \geq 3$ and $3^*,k,k+j$, where $j \geq 4$. There are also trees not belonging to those families. For example, a tree of order 28 does not belong to any infinite family like the previous ones we found. You can see it in Figure \ref{fig:tree_3*_3_5}.

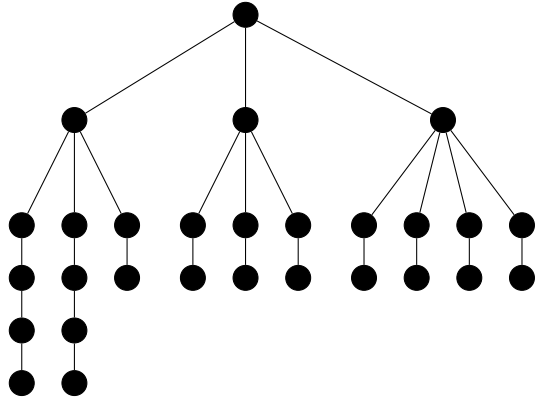
\begin{figure}[H]
    \centering
    \begin{tikzpicture}[scale=0.7]
        \tikzstyle{black}=[fill=black, draw=black, shape=circle]
		\node [style=black] (0) at (0, 2) {};
		\node [style=black] (1) at (1, 2) {};
		\node [style=black] (2) at (2, 2) {};
		\node [style=black] (3) at (0, 3) {};
		\node [style=black] (4) at (1, 3) {};
		\node [style=black] (5) at (2, 3) {};
		\node [style=black] (6) at (1, 5) {};
		\node [style=black] (7) at (3.25, 2) {};
		\node [style=black] (8) at (4.25, 2) {};
		\node [style=black] (9) at (5.25, 2) {};
		\node [style=black] (10) at (3.25, 3) {};
		\node [style=black] (11) at (4.25, 3) {};
		\node [style=black] (12) at (5.25, 3) {};
		\node [style=black] (13) at (4.25, 5) {};
		\node [style=black] (14) at (6.5, 2) {};
		\node [style=black] (15) at (7.5, 2) {};
		\node [style=black] (16) at (8.5, 2) {};
		\node [style=black] (17) at (6.5, 3) {};
		\node [style=black] (18) at (7.5, 3) {};
		\node [style=black] (19) at (8.5, 3) {};
		\node [style=black] (20) at (8, 5) {};
		\node [style=black] (21) at (4.25, 7) {};
		\node [style=black] (22) at (0, 1) {};
		\node [style=black] (23) at (0, 0) {};
		\node [style=black] (24) at (9.5, 2) {};
		\node [style=black] (25) at (9.5, 3) {};
		\node [style=black] (26) at (1, 1) {};
		\node [style=black] (27) at (1, 0) {};

		\draw (3) to (0);
		\draw (4) to (1);
		\draw (5) to (2);
		\draw (6) to (5);
		\draw (6) to (4);
		\draw (6) to (3);
		\draw (10) to (7);
		\draw (11) to (8);
		\draw (12) to (9);
		\draw (13) to (12);
		\draw (13) to (11);
		\draw (13) to (10);
		\draw (20) to (17);
		\draw (20) to (18);
		\draw (20) to (19);
		\draw (17) to (14);
		\draw (18) to (15);
		\draw (19) to (16);
		\draw (21) to (20);
		\draw (21) to (13);
		\draw (21) to (6);
		\draw (0) to (22);
		\draw (22) to (23);
		\draw (20) to (25);
		\draw (25) to (24);
		\draw (1) to (26);
		\draw (27) to (26);

\end{tikzpicture}

    \caption{An exceptional tree of order 28 that has non-log-concave independence polynomial.}
    \label{fig:tree_3*_3_5}
\end{figure}

\[
I(T;x) =\!\begin{multlined}[t]
     x^{15} + 55x^{14} + 3139x^{13} + 24020x^{12} + 86526x^{11} + 187731x^{10} + \\272268x^9 +278417x^8 + 206422x^7 + 112284x^6 + 44772x^5 + 12910x^4 \\+ 2613x^3 + 351x^2 + 28x + 1,
    \end{multlined}
\]
where the non-log-concavity is demonstrated by the coefficient of $x^{14}$: \\$55^2=3025<3139$.\\

In all our previous counterexamples the log-concavity was violated at the $\alpha(G)-1$ coefficient.
Nevertheless, there are trees with broken log-concavity at the $\alpha(G)-2$ coefficient. For instance, see Figure \ref{fig:alpha-2_coefficient}.\\
\begin{figure}[H]
    \centering
    \begin{tikzpicture}[scale=0.7]
        \tikzstyle{black}=[fill=black, draw=black, shape=circle]
        
		\node [style=black] (0) at (0, 2) {};
		\node [style=black] (1) at (1, 2) {};
		\node [style=black] (2) at (2, 2) {};
		\node [style=black] (3) at (0, 3) {};
		\node [style=black] (4) at (1, 3) {};
		\node [style=black] (5) at (2, 3) {};
		\node [style=black] (6) at (1, 5) {};
		\node [style=black] (7) at (3.25, 2) {};
		\node [style=black] (8) at (4.25, 2) {};
		\node [style=black] (9) at (5.25, 2) {};
		\node [style=black] (10) at (3.25, 3) {};
		\node [style=black] (11) at (4.25, 3) {};
		\node [style=black] (12) at (5.25, 3) {};
		\node [style=black] (13) at (4.75, 5) {};
		\node [style=black] (14) at (7.5, 2) {};
		\node [style=black] (15) at (8.5, 2) {};
		\node [style=black] (16) at (9.5, 2) {};
		\node [style=black] (17) at (7.5, 3) {};
		\node [style=black] (18) at (8.5, 3) {};
		\node [style=black] (19) at (9.5, 3) {};
		\node [style=black] (20) at (9, 5) {};
		\node [style=black] (21) at (6.875, 7) {};
		\node [style=black] (24) at (10.5, 2) {};
		\node [style=black] (25) at (10.5, 3) {};
		\node [style=black] (26) at (6.25, 2) {};
		\node [style=black] (27) at (6.25, 3) {};
  
		\node [style=black] (28) at (11.75, 2) {};
		\node [style=black] (29) at (11.75, 3) {};
		\node [style=black] (30) at (12.75, 2) {};
		\node [style=black] (31) at (12.75, 3) {};
		\node [style=black] (32) at (13.75, 2) {};
		\node [style=black] (33) at (13.75, 3) {};
		\node [style=black] (34) at (14.75, 2) {};
		\node [style=black] (35) at (14.75, 3) {};
		\node [style=black] (36) at (13.25, 5) {};

		\draw (3) to (0);
		\draw (4) to (1);
		\draw (5) to (2);
		\draw (6) to (5);
		\draw (6) to (4);
		\draw (6) to (3);
		\draw (10) to (7);
		\draw (11) to (8);
		\draw (12) to (9);
		\draw (13) to (12);
		\draw (13) to (11);
		\draw (13) to (10);
		\draw (20) to (17);
		\draw (20) to (18);
		\draw (20) to (19);
		\draw (17) to (14);
		\draw (18) to (15);
		\draw (19) to (16);
		\draw (21) to (20);
		\draw (21) to (13);
		\draw (21) to (6);
		\draw (20) to (25);
		\draw (25) to (24);
		\draw (13) to (27);
		\draw (27) to (26);
		\draw (28) to (29);
		\draw (30) to (31);
		\draw (32) to (33);
		\draw (34) to (35);
		\draw (36) to (29);
		\draw (36) to (31);
		\draw (36) to (33);
		\draw (36) to (35);
		\draw (36) to (21);
  
\end{tikzpicture}

    \caption{A tree with broken log-concavity at the $\alpha(G)-2$ coefficient.}
    \label{fig:alpha-2_coefficient}

\end{figure}
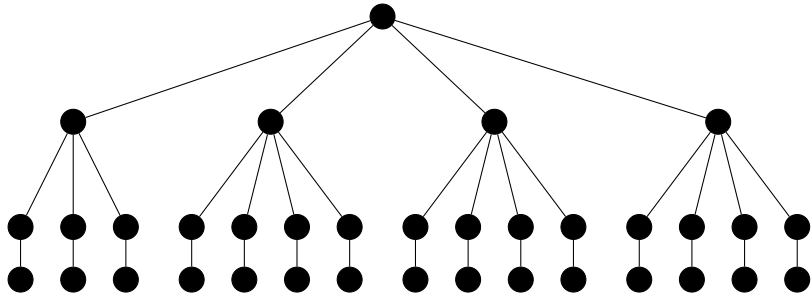

\[
I(T;x) =\!\begin{multlined}[t]
x^{19}+71x^{18}+1989x^{17}+60773x^{16}+458294x^{15}+1773294x^{14}\\+4355940x^{13}+7485954x^{12}+9489531x^{11}+9151478x^{10}+6835096x^9\\+3989058x^8+1822331x^7+648445x^6+177534x^5+36596x^4\\+5480x^3+561x^2+35x+1,
    \end{multlined}
\]
where the non-log-concavity is demonstrated by the coefficient of $x^{17}$: \\$1989^2=3956121<4314883=71\cdot 60773$.\\

It motivates the following.

\begin{conjecture}
    The log-concavity may be broken at the $\alpha(G)-k$ coefficient for arbitrarily $k \in \{1,2,\cdots, \alpha(G) -1\}$.
\end{conjecture}


\begin{thebibliography}{50}

\bibitem{AlaviMaldeSchwenkErdos1987} Alavi Y, Malde PJ, Schwenk AJ, Erd{\"o}s P. The vertex independence sequence of a graph is not constrained. Congressus Numerantium. 1987;\textbf{58}:15-23

\bibitem{Arocha1984} Arocha JL. Propriedades del polinomio independiente de un grafo. Revista Ciencias Matematicas. 1984;\textbf{3}:103-110

\bibitem{BallGalvinHyryWeingartner2021} Ball T, Galvin D, Hyry C, Weingartner K. Independent set and matching permutations. Journal of Graph Theory. 2022;\textbf{99}:40-57. \url{DOI: 10.1002/jgt.22724}

\bibitem{BasitGalvin2021} Basit A, Galvin D. On the independent set sequence of a tree. The Electronic Journal of Combinatorics. 2021;\textbf{28(3)}:P3.23. \url{DOI: 10.37236/9896}

\bibitem{BeatonBrown2022} Beaton I, Brown JI. On the unimodality of domination polynomials. Graphs and Combinatorics. 2022;\textbf{38}:90. \url{DOI: 10.1007/s00373-022-02487-x}

\bibitem{BenderCanfield1996} Bender EA, Canfield ER. Log-concavity and related properties of the cycle index polynomials. Journal of Combinatorial Theory A. 1996;\textbf{74}:57-70

\bibitem{BhattacharyaKahn2013} Bhattacharya A, Kahn J. A bipartite graph with non-unimodal independent set sequence. Electronic J. Combinatorics 2013;\textbf{20(4)};P11

\bibitem{Bodlaender1996} Bodlaender HL. A linear time algorithm for finding tree-decompositions of small treewidth. SIAM. 1996;\textbf{25}:1305-1317

\bibitem{BorosMoll1999} Boros G, Moll VH. A sequence of unimodal polynomials. Journal of Mathematical Analysis and Applications. 1999;\textbf{237}:272-287

\bibitem{Brenti1990} Brenti F. Unimodal, log-concave and Polya frequency sequences in Combinatorics. Memories Amer, Math, Soc. 1990;\textbf{108}:729-756

\bibitem{Brenti1994} Brenti F. Log-concave and unimodal sequences in algebra, combinatorics, and geometry: An update, “Jerusalem Combinatorics’93”. Contemporary Mathematics. 1994;\textbf{178}:71-89

\bibitem{BrownColbourn1994} Brown JI, Colbourn CJ. On the logconcavity of reliability and matroidal sequences. Advances in Applied Mathematics. 1994;\textbf{15}:114-127

\bibitem{ChudnovskySeymour2006} Chudnovsky M, Seymour P. The roots of the independence polynomial of a claw-free graph. Journal of Combinatorial Theory, Series B. 2006;\textbf{97}:350-357. \url{DOI:10.1016/j.jctb.2006.06.001}

\bibitem{Dukes2002} Dukes WMB. On a unimodality conjecture in matroid theory. Discrete Mathematics and Theoretical Computer Science. 2002;\textbf{5}:181-190

\bibitem{Ferrin2014} Ferrin GM. Independence polynomials [Thesis]. Columbia, South Carolina: University of South Carolina; 2014

\bibitem{Galvin2011} Galvin D. REGS 2011. Available from: \url{https://faculty.math.illinois.edu/west/regs/stasetseq.html}.

\bibitem{GutmanHarary1983} Gutman I, Harary F. Generalizations of the matching polynomial. Utilitas Mathematica. 1983;\textbf{24}:97-106

\bibitem{Hamidoune1990} Hamidoune YO. On the number of independent k-sets in a claw-free graph. Journal of Combinatorial Theory B. 1990;\textbf{50}:241-244

\bibitem{Hoede_Li1994} ] Hoede C, Li X. Clique polynomials and independent set polynomials of graphs. Discrete Mathematics. 1994;\textbf{125}:219-228

\bibitem{Horrocks2002}Horrocks DGC. The numbers of dependent k-sets in a graph are log-concave. Journal of Combinatorial Theory, Series B;2002;\textbf{84}:180-185. \url{DOI: 10.1006/jctb.2001.2077}

\bibitem{KadrawiLevitYosefMizrachi2023} Kadrawi O, Levit VE, Yosef R, Mizrachi M. On Computing of Independence Polynomials of Trees, A chapter in Recent Research in Polynomials, IntechOpen 2023. \url{https://www.intechopen.com/chapters/1130709} ISBN 978-1-83769-496-9

\bibitem{LevitMandrescu2003} Levit VE, Mandrescu E. On unimodality of independence polynomials of some well-covered trees. Discrete Mathematics and Theoretical Computer Science. 2003;\textbf{4}:237-256

\bibitem{LevitMandrescu2004} Levit VE, Mandrescu E. Very well-covered graphs with log-concave independence polynomials. Carpathian J. Math. 2004;\textbf{20}:73-80

\bibitem{LevitMandrescu2005}  Levit VE, Mandrescu E. The independence polynomial of a graph - a survey. In: Proceedings of the 1st International Conference on Algebraic Informatics. Thessaloniki: Aristotle University of Thessaloniki. 2005;231-252

\bibitem{LevitMandrescu2006}
Levit VE, Mandrescu E. Partial unimodality for independence
polynomials of K\"{o}nig-Egerv\'{a}ry graphs. Congressus Numerantium,179:
2006;109-119.

\bibitem{Radcliffe} Radcliffe AJ. Personal communication (mentioned in Ball T, Galvin D, Hyry C, Weingartner K. Independent set and matching permutations. Journal of Graph Theory. 2022;\textbf{99}:40-57.)

\bibitem{Schwenk1981} Schwenk AJ. On unimodal sequences of graphical invariants. Journal of Combinatorial Theory B. 1981;\textbf{30}:247-250

\bibitem{Schwenk2015} Schwenk AJ. Smaller bipartite graphs with non-unimodal
independent set sequence \url{https://www.researchgate.net/publication/281745302_Smaller_bipartite_graphs_with_non-unimodal_independent_set_sequence}

\bibitem{Stanley1989} Stanley RP. Log-concave and unimodal sequences in algebra, combinatorics, and geometry. Annals of the New York Academy of Sciences. 1989;\textbf{576}:500-535

\bibitem{Tittmann2021} Tittmann P. Graph Polynomials: The Eternal Book. Productivity Press; 2021

\bibitem{YosefMizrachiKadrawi2021} Yosef R, Mizrachi M, Kadrawi O. On unimodality of independence polynomials of trees. 2021. Available from: \url{https://arxiv.org/pdf/2101.06744v3.pdf} 

\bibitem{XieFengXu2023} Xie YT, Feng YD, Xu SJ. A relation between the cube polynomials of partial cubes and the clique polynomials of their crossing graphs. 2023. \url{https://arXiv:2303.14671v1}

\end{thebibliography}
\end{document}